\documentclass[a4paper, 11pt, reqno]{amsart}
\usepackage{amssymb}
\usepackage{a4wide}
\usepackage{verbatim}

\usepackage{enumerate}

\newtheorem{theorem}{Theorem}
\newtheorem{cor}{Corollary}
\theoremstyle{definition}
\newtheorem{defn}{Definition}
\newtheorem{lemma}{Lemma}

\newtheorem*{Dickinson}{Theorem (Dickinson)}

\newcommand{\mb}[1]{\mathbb{#1}}
\newcommand{\mbf}[1]{\mathbf{#1}}
\newcommand{\mcal}[1]{\mathcal{#1}}

\newcommand{\W}{W_0\left( m,n;\psi \right)}

\begin{document}
\title{The metrical theory  of simultaneously small linear forms. }

\author[M.~Hussain]{Mumtaz Hussain}
\address{Mumtaz Hussain, Department of Mathematics, University of
  York, Heslington, York, YO10 5DD, United Kingdom}
\email{mh577@york.ac.uk}

\author[J.~Levesley]{Jason Levesley}
\address{Jason Levesley, Department of Mathematics, University of
  York, Heslington, York, YO10 5DD, United Kingdom}
\email{jl107@york.ac.uk}

\subjclass[2000]{Primary 11J83; Secondary 11J13, 11K60}

\keywords{Diophantine approximation, Khintchine type theorems,
system linear forms, Hausdorff measure.}

\begin{abstract}
In this paper we investigate  the metrical theory of
Diophantine approximation associated with linear forms that are
simultaneously small for infinitely many integer vectors;  i.e.
forms which are close to the origin.   A complete Khintchine--Groshev
type theorem is established, as well as its Hausdorff measure
generalization. The latter implies the complete Hausdorff dimension theory.
\end{abstract}

\maketitle


\section{Introduction}

Let $\psi:\mb{R}^+\to\mb{R}^+  $ be a real positive
decreasing function with $\psi(r)\to{}0$ as $r\to\infty$.
Such a function will be refereed to as an \emph{approximation}
function. An $m\times n$ matrix $X=(x_{ij})\in\mb{R}^{mn}$ is said
to be \emph{$\psi$--approximable} if the system of inequalities
\[
|q_{_{1}}x_{_{1}i}+q_{_{2}}x_{_{2}i}+\dots+q_{m}x_{mi}| \leq \psi(|\mbf{q}|)\text{\quad{}for\quad}(1\leq
i\leq n),
\]
is satisfied for infinitely many $\mbf{q}\in \mb{Z} ^{m}\setminus\{0\}$.
Here and throughout $|\mbf{q}|$ will
denote the supremum norm of the vector $\mbf{q}$. Specifically,
$\left\vert \mbf{q}\right\vert =\max \left\{ \left\vert
q_{_{1}}\right\vert ,\left\vert q_{_{2}}\right\vert ,\dots,\left\vert
q_{_{m}}\right\vert \right\}$.
The system
$q_{_{1}}x_{_{1}i}+q_{_{2}}x_{_{2}i}+\dots+q_{m}x_{mi}$ of $n$ linear
forms in $m$ variables $q_{_{1}},q_{_{2}},\dots,q_{_{m}}$ will be
written more concisely as $\mbf{q}X$, where the matrix $X$ is
regarded as a point in $ \mb{R}^{mn}.$  It is easily verified that
$\psi$--approximability is not affected under translation by integer
vectors and we can therefore restrict attention to the  unit cube
$\mb{I}^{mn}:= [ -\frac{1}{2},\frac{1}{2}] ^{mn}$. The set of
$\psi$--approximable points in $\mb{I}^{mn} $ will be denoted by
$W_0( m,n;\psi)$; 
\[
W_0( m,n;\psi) :=\{ X\in \mb{I}^{mn}:|\mbf{q}X|<\psi(|\mbf{q}|)\text{\ for i.m.\ }\mbf{q}\in \mb{Z}^{m}\setminus \{\mbf{0}\}\},
\]
where `i.m.' means `infinitely many'.
In the case when $\psi(r)= r^{-\tau}$ for some $ (\tau>0)$ we shall write $W_0\left( m,n;\tau
\right)$ instead of $W_0\left( m,n;\psi \right)$.

It is worth relating the above to the set of $\psi$--well approximable matrices as is often studied in classical Diophantine approximation.
In such a setting studying the metric structure of the $\limsup$-set
\[
W( m,n;\psi)=\{ X\in \mb{I}^{mn}:\| \mbf{q}X\| <\psi(|\mbf{q}|)\text{ for i.m. \ }  \mbf{q}\in
\mb{Z}^{m}\setminus \{\mbf{0}\}\},
\]
where $\|x\|$ denotes the distance of $x$ to the nearest
integer vector, is a central problem and the theory is well
established, see for example \cite{geo} or \cite{BDV}. Probably the main result in this setting is the Khintchine-Groshev theorem which gives
an elegant answer to the question of the size of the $W(m,n;\psi)$. The result links the measure of the set to the convergence or otherwise
of a series that depends only on the approximating function and is the template for many results in the field of metric number theory.  It is clear
then that the set $\W$ is an analogue of $W(m,n; \psi)$
with $|\cdot|$ replacing $\|\cdot\|$. The aim of this paper is to obtain the complete metric theory for the set $W_0(m,n;\psi)$.

It is readily verified that $W_0( 1,n;\psi) = \{0\}$ as any $x=(x_1,x_2,\dots,x_n)\in{}W_0(1,n;\psi)$
must satisfy the inequality
$\left\vert qx_{j}\right\vert < \psi(q) $  infinitely often. As
$\psi(q)\to 0$ as $q\to\infty$ this is only possible if
$x_j=0$ for all $j=1,2,\dots,n.$ 
Thus when $m=1$ the set $W_0(1, n; \psi)$ is a singleton and must have both zero measure and dimension.
We will therefore assume that $m\geq 2.$

Before giving the main results of this paper we include a brief review of some of the work done previously on
the measure theoretic structure of  $W_0(m,n;\psi)$.

The first result is due to Dickinson
\cite{Dickinson}.
\begin{Dickinson}\label{hd}
When $\tau> \frac{m}{n}-1$ and $m\geq2$
\[
\dim(W_0(m,n;\tau))=(m-1)n+ \frac{m}{\tau+1},
\]
and when $0<\tau\leq \frac{m}{n}-1,$
\[
\dim(W_0(m,n;\tau))=mn.
\]
\end{Dickinson}
It turns out that Dickinson's original result is false when $m
\leq n$.  The correct statement is given in Corollary $5$ which
is a consequence of Theorem \ref{Mumtaz2} proved below.
To the best of our knowledge the
only other result is due to Kemble \cite{kemble} who established a
Khintchine--Groshev type theorem for  $W_0(m,1;\psi)$ under various
conditions on the approximating function. We shall remove these
conditions and prove the precise analogue of  the
Khintchine--Groshev theorem for  $W_0(m,n;\psi)$. Finally, it is
worth mentioning that the set is not only of number theoretic
interest but appears naturally in operator theory, see
\cite{perturbation} for further details.

\emph{Notation.} To simplify notation the symbols $\ll$ and
$\gg$ will be used to indicate an inequality with an unspecified
positive multiplicative constant. If $a\ll b$ and $a\gg b$  we write
$a\asymp b$, and say that the quantities $a$ and $b$ are comparable.
For  a set $A$, $|A|_k$ will be taken to mean the $k-$dimensional Lebesgue
measure of the set $A$.

\section{Statement of Main Results}

The results of this paper depend crucially on the choice of $m$ and $n$. We
shall see that when $m>n$, the metric theory is `independent' and
for this particular case Dickinson's dimension result is correct.
When $m\leq n$ the measure results are dependent on the independent
case. Dickinson's result for this particular case is incorrect and we provide
the correct result.

In the following $\mcal{H}^{f}$ denotes
$f$-dimensional Hausdorff measure which will be defined fully in \S\ref{hm}. Given an
approximating function $\psi$ let $\Psi(r):=\frac{\psi(r)}{r}$.

\begin{theorem}\label{thm1}
Let $m>n$ and $\psi$ be an approximating function. Let $f$ be a
dimension function such that $r^{-mn}f(r)$ is monotonic and
$r^{-(m-1)n}f(r)$ is increasing. Then
\[
\mcal{H}^{f}\left(W_0\left( m,n;\psi \right) \right)=
\left\{\begin{array}{lll} 0&\mbox{if}& \sum \limits_{r=1}^{\infty}f(\Psi (r))\Psi(r)^{-(m-1)n}r^{m-1}<\infty, \\
    \mcal{H}^f(\mb{I}^{mn})&{if}& \sum \limits_{r=1}^{\infty}f(\Psi
    (r))\Psi(r)^{-(m-1)n}r^{m-1}=\infty.
\end{array}\right.
\]
\end{theorem}
The requirement that $r^{-mn}f(r)$ be monotonic is a natural and not particularly restrictive
condition. Note that if the dimension
function $f$ is such that $r^{-mn}f(r)\to \infty$ as $r\to 0$ then
$\mcal{H}^f(\mb{I}^{mn})=\infty$ and Theorem \ref{thm1} is the
analogue of the classical result of Jarn\'{\i}k (see \cite{Jarnik}).

Theorem \ref{thm1} implies analogues of both the Lebesgue and Hausdorff measure results
familiar from classical Diophantine approximation. In the case when $f(r) := r^{mn}$
the Hausdorff measure $\mcal{H}^f$  is simply standard Lebesgue measure supported on
$\mb{I}^{mn}$ and the result is the natural analogue of the
Khintchine--Groshev theorem for $ W_0\left( m,n;\psi\right )$.
\begin{cor}\label{cor1}
Let $m>n$ and $\psi $
be an approximating function, then
\begin{equation*}
|W_0\left( m,n;\psi \right)|_{mn} =
\begin{cases}
    0\hspace{1cm} if\hspace{1cm} \sum \limits_{r=1}^{\infty}\psi (r)^{n}r^{m-n-1}<\infty, \\
    1\hspace{1cm} if\hspace{1cm} \sum \limits_{r=1}^{\infty}\psi
    (r)^{n}r^{m-n-1}=\infty.
\end{cases}
\end{equation*}
\end{cor}

If we now set  $f:r\to
r^s(s>0)$ then Theorem \ref{thm1} reduces to
the following $s$-dimensional Hausdorff measure statement which is more
discriminating then the Hausdorff dimension result of
Dickinson.
\begin{cor}\label{smesmgn}Let $m>n$ and $\psi$ be an approximating function. Let $s$ be  such that
$(m-1)n<s\leq mn.$ Then,
\[
\mcal{H}^{s}\left( {W_0}\left( m,n;\psi \right) \right)=
\left\{\begin{array}{lll} 0&\mbox{if}& \sum \limits_{r=1}^{\infty}\Psi(r)^{s-(m-1)n}r^{m-1}<\infty , \\
    \mcal{H}^s(\mb{I}^{mn})&{if}& \sum
    \limits_{r=1}^{\infty}\Psi(r)^{s-(m-1)n}r^{m-1}=\infty.
\end{array}\right.
\]
\end{cor}
Under the conditions of Corollary~\ref{smesmgn} it follows from
the definition of Hausdorff dimension that
\[
\dim\left(W_0\left( m,n;\psi \right) \right)=\inf \left\{s:\sum
\limits_{r=1}^{\infty}\Psi(r)^{s-(m-1)n}r^{m-1}<\infty \right\},
\]
and in particular the following
dimension result for $W_0(m,n;\tau)$ holds.
\begin{cor}\label{dimmln} Let $m>n$ and $\tau>\frac{m}{n}-1$ then
\[
\dim\left(W_0\left( m,n;\tau \right)
\right)=(m-1)n+\frac{m}{\tau+1}.
\]
\end{cor}

Theorem~\ref{thm1} establishes the metric theory for $W_0\left(
m,n;\psi \right) $ when $m>n.$ 
For the cases when $m\leq{}n$ the statement of Theorem~\ref{thm1} changes somewhat. 
The sum which determines the $f$--measure remains the same
but the conditions on the dimension functions are different. 
This is due to the fact that set $W(m,n;\psi)$ can be shown to lie in a manifold $\Gamma\subset\mathbb{R}^{mn}$ of dimension $(m-1)(n+1)$, a fact we prove later in \S\ref{proofThm2}.
In light of this remark, an upper bound for
$\dim{}\W$ follows immediately. More specifically,
\[
\dim W_0(m,n;\psi)\leq (m-1)(n+1).
\]
 
\begin{theorem}\label{Mumtaz2}
Let $m\leq n$ and  $\psi$ be an approximating function. Let $f$ and
$g$ be dimension functions with $g(r)=r^{-(n-m+1)(m-1)}f(r)$. Assume that
$r^{-(m-1)(n+1)}f(r)$ is monotonic and $r^{-(m-1)n}f(r)$
increasing. Then $\mcal{H}^{f}(W_0(m,n;\psi)=0$ if 
\[
	\sum_{r=1}^\infty{}f(\Psi(r))\Psi(r)^{-(m-1)n}r^{m-1}<\infty.
\]

If
\[
		\sum_{r=1}^\infty{}f(\Psi(r))\Psi(r)^{-(m-1)n}r^{m-1}=\infty,
\]
then 
\[
	\mcal{H}^{f}(W_0(m,n;\psi) = 
	    \begin{cases}
			\infty & \text{\quad{}if\quad{}}  r^{-(m-1)(n+1)}f(r)\to\infty \text{ as } r\to{}0, \\
			\mathcal{H}^f(\Gamma) & \text{\quad{}if\quad{}} r^{-(m-1)(n+1)}f(r)\to{}C \text{ as } r\to{}0,
	    \end{cases}                          
\]
where $C>0$ is some fixed constant.  
\end{theorem}
It is worth noting that for dimension functions $f$ such that $r^{-(m-1)(n+1)}f(r)\to{}C>0$ as $r\to{}0$ the measure $\mathcal{H}^f$ 
is comparable to standard $(m-1)(n+1)$-dimensional Lebesgue measure and 
in the case when $f(r)=r^{(m-1)(n+1)}$, we obtain the following analogue of the 
Khintchine-Groshev theorem.
\begin{cor}\label{kgmleqn}
  Let $m\leq n$ and $\psi$ be an approximating function and assume that the conditions of
Theorem~\ref{Mumtaz2} hold for the dimension function $f(r):=r^{(m-1)(n+1)}$. Then
\[
  \mu(W_0 (m,n;\psi))=\begin{cases}
    0\hspace{1cm} if\hspace{1cm} \sum \limits_{r=1}^{\infty}\psi (r)^{m-1}<\infty, \\
    1\hspace{1cm} if\hspace{1cm} \sum \limits_{r=1}^{\infty}\psi (r)^{m-1}=\infty,
	\end{cases}
\]
where $\mu$ is the normalised measure on the manifold $\Gamma$.
\end{cor}
As above, if we set $f(r)=r^s$ we obtain the $m\leq n$ analogue
of Corollary~\ref{smesmgn}.
\begin{cor}\label{smesln}Let $m\leq n$  and $\psi$ be an approximating function. Let $s$ be such that $(m-1)n<s\leq
(m-1)(n+1)$ and let  $g:r\to r^{s-(n-(m-1))(m-1)}$ be a dimension
function.  Then,
\[
  \mcal{H}^{s}(W_0(m,n;\psi))=\left\{\begin{array}{lll} 0&\mbox{if}& \sum \limits_{r=1}^{\infty}\Psi(r)^{s-(m-1)n}r^{m-1}<\infty, \\
    \mcal{H}^s(\Gamma)&{if}& \sum \limits_{r=1}^{\infty}\Psi(r)^{s-(m-1)n}r^{m-1}=\infty.
\end{array}\right.
\]
\end{cor}
Further, under the same conditions as Corollary~\ref{smesln}, but with the approximation function $\psi(x)= x^{-\tau}$, we have the following result:
\begin{cor}\label{dimWmleqn}
Let $ m\leq n$ and $\tau> \frac{m}{m-1}-1.$ Then
\[
\dim(W_0(m,n;\tau))=(m-1)n+ \frac{m}{\tau+1},
\]
and when $0<\tau\leq \frac{m}{m-1}-1,$
\[
\dim(W_0(m,n;\tau))=(m-1)(n+1).
\]
\end{cor}

The paper is organized as follows.  In Section~\ref{aux},
we give the definitions of Hausdorff measure and ubiquity, which is the
main
tool for proving Theorem~\ref{thm1},
in a manner appropriate to the setting of this paper. 
Section~\ref{aux} also includes the statement of the `Slicing' lemma (Lemma~\ref{slicing}) which is used to prove Theorem~\ref{Mumtaz2}. The paper continues with the proof of Theorem~\ref{thm1}
in \S~\ref{sthm1}. As is common when proving such `zero-full' results the proof is split into two parts; the convergence case and the divergence case. We conclude the paper with the proof of Theorem~\ref{Mumtaz2}.

\section{Basic Definitions and Auxiliary Results}\label{aux}

In this section we give definitions of some fundamental concepts along
with some auxiliary results which will be needed in the proofs of Theorems~\ref{thm1} and \ref{Mumtaz2}.

\subsection{Hausdorff Measure and Dimension}
\label{hm}

Below we give a brief introduction to Hausdorff $f$-measure and dimension. For further details see \cite{Falc2}.

Let $f:\mb{R}^+\rightarrow \mb{R}^+$ be an increasing continuous
function such that $f(r)\to 0$ as $r\to 0$. Such a function $f$ is
referred to as a \emph{dimension function}. We are now in a position to define the
Hausdorff $f$--measure $\mathcal{H}^f(X) $ of a set $X\subset \mb{R}^n$.

Let $B$ be a (Euclidean) ball in $\mb{R}^n$. That is a set of the form
\[
    B = \{ x\in\mb{R}^n: |x - c |_2 < \delta \}
\]
for some $c\in\mb{R}^n$ and some $\delta>0$. The \textit{diameter} $\mathrm{diam}(B)$ of $B$ is
\[
    \mathrm{diam}(B) := \sup\{ |x-y|_2: x,y\in{}B \}.
\]
Now for any $\rho>0$ a countable collection $\{B_i\}$ of balls in
$\mb{R}^n$ with diameters $\mathrm{diam} (B_i)\le \rho$ such that $X\subset
\bigcup\limits_i B_i$ is called a $\rho$--cover for $X$. Define
\begin{equation*}
\mcal{H}_\rho^f(X)=\inf\left\{\sum\limits_if(\mathrm{diam}(B_i))\;:\;\{B_i\}%
\mbox{ is a $\rho$-cover for }X\right\},
\end{equation*}
where the infimum is taken over all possible
$\rho$--covers of $X$. The Hausdorff $f$-measure of $X$ is defined
to be
\[
\mcal{H}^f(X)=\lim\limits_{\rho\to 0}\mcal{H}_\rho^f(X).
\]
In the particular case when $f(r):=r^s$  $(s>0)$, we write
$\mcal{H}^s (X)$ for  $\mcal{H}^f$ and the measure
 is refereed to as $s$--dimensional Hausdorff measure. The
Hausdorff dimension of a set $X$ is denoted by $\dim(X)$ and is
defined  as follows,
\[
\dim(X):=\inf\{s\in \mb{R}^+\;:\; \mcal{H}^s(X)=0\}=\sup\{s\in
\mb{R}^+\;:\; \mcal{H}^s(X)=\infty\}.
\]
Note that the value of $\dim(X)$ is unique. At the critical exponent $s = \dim X$ the quantity
${\mcal H}^s(X) $ is either zero, infinite or strictly positive and
finite. In the latter case; i.e. when
\[
0< {\mcal H}^s(X)<\infty,
\]
the set $X$ is said to be an $s$-set.

\subsection{Ubiquitous Systems}\label{ubiq}

To make this article as self contained as possible we
describe the main tool used in proving the divergence
part of Theorem \ref{thm1}, the idea of a locally ubiquitous system. The set-up presented below is simplified
for the current problem. The general framework is much more abstract and full details can
be found in \cite{geo} and \cite{BDV}.

Let $\Re =\left\{
R_{\mbf{q} }:\mbf{q} \in \mb{Z}^m\setminus\{\mbf{0}\}\right\} $ be the
family of subsets $R_{\mbf{q}
}:=\{X\in\mb{I}^{mn}:\mbf{q}X=\mbf{0}\}$. The sets $R_{\mbf{q}}$
will be referred to as \emph{resonant sets}.
Let the function
$\beta:\mb{Z}^m\setminus\{\mbf{0}\}\to\mb{R}^+:\mbf{q}\to |\mbf{q}|$
attach a weight to the resonant set $R_{\mbf{q}}.$ Now, given an
approximating function $\psi$ and $R_{\mbf{q}
}$,
let
\begin{equation*}
\Delta\left( R_{q},\Psi (\vert\mbf{q}%
\vert )\right) :=\left\{ X\in \mb{I}^{mn}: \text{dist}\left(
X,R_{q}\right) \leq \frac{\psi (\left\vert \mbf{q}\right\vert )}{%
\left\vert \mbf{q}\right\vert }\right\}
\end{equation*}
where $\mathrm{dist}(X,R_{\mbf{q}}):=\inf \{|X-Y|:Y\in
R_{\mbf{q}}\}.$ Thus $\Delta\left( R_{q},\Psi (\vert\mbf{q}%
\vert )\right)$ is a $\Psi$--neighbourhood of $R_{\mbf{q}}$.
Notice that in the case when the resonant sets are
points the sets $\Delta\left( R_{q},\Psi (\vert\mbf{q}%
\vert )\right)$ are simply balls centred at resonant points.

Let
\[
\Lambda(m,n;\psi)
=\{X\in\mb{I}^{mn}:X\in \Delta\left( R_{q},\Psi (\vert\mbf{q}%
\vert )\right) \text{for i.m. }
\mbf{q}\in\mb{Z}^m\setminus\{\mbf{0}\}\}.
\]
The set $\Lambda(m,n; \psi)$ is a `limsup' set. It consists entirely
of points in $\mb{I}^{mn}$ which lie in infinitely many of the
sets $\Delta\left( R_{q},\Psi (\vert\mbf{q}%
\vert )\right)$.
This is apparent if we restate
$\Lambda(m,n;\psi)$ in a manner
which emphasises its limsup structure.

Fix $k>1$ and for any
$t\in\mb{N}$,
define
\begin{equation}
\label{limsup1}
    \Delta(\psi,t):= 
\bigcup_{
k^{t-1}\leq |\mbf{q}|\leq{}k^t}\Delta(
R_{\mbf{q} },\Psi(|\mbf{q}|)).
\end{equation}
It follows that
\begin{equation}
\label{limsup2}
\Lambda(m,n;\psi) = 
\limsup_{t\rightarrow \infty }%
\Delta( \psi ,t) =
\bigcap_{N=1}^{\infty}\bigcup_{t=N}^\infty{}\Delta( \psi,t).
\end{equation}
The key point by which ubiquity will be utilised is in the fact that the sets $W_0(m,n;\psi)$ and $\Lambda(m,n;\psi)$ actually coincide.

We now move onto the formal definition of a locally ubiquitous system. As stated above 
the definition given below is in a much simplified form suitable to the problem at hand. In the more abstract setting given in \cite{BDV} there are specific conditions on both the measure on the ambient space and its interaction with neighbourhoods of the resonant set which must be shown to  hold. These conditions are not stated below as they hold trivially for Lebesgue measure, the measure on our ambient space $\mathbb{I}^{mn}$, and stating the conditions would complicate the discussion somewhat. Never the less, the reader should be aware that in the more abstract notion of ubiquity these extra conditions exist and need to be be established.

Let $\rho :\mb{R}^{+}\rightarrow\mb{R}^{+}$ be a
function with $\rho \left( r\right) \rightarrow 0$ as $%
r\rightarrow \infty $ and let%
\[
\Delta \left( \rho ,t\right) :=
\underset{\mbf{q} \in J\left(
t\right) }{\bigcup }\Delta\left( R_{\mbf{q} },\rho \left( k^t\right)
\right)
\]
where $J(t)$ is defined to be the set
\[
 J(t):= \left\{\mbf{q}\in\mb{Z}^m\setminus\{\mbf{0}\}: |\mbf{q}|\leq
k^t\right\}
\]
for a fixed constant $k>1$.

\begin{defn}
Let
$B:=B\left( X,r\right) $ be an arbitrary ball with centre $X\in
\mb{I}^{mn} $ and $r\leq r_{o}.$ Suppose there exists a function
$\rho$ and an absolute constant $\kappa>0$ such that
\[
|B\cap \Delta \left( \rho ,t\right)|_{mn} \geq \kappa | B|_{mn}
\text{ for }t\geq t_{o}\left( B\right).
\]
Then the pair $\left( \Re ,\beta \right) $ is said to be a
\emph{locally ubiquitous} system relative to $\left( \rho ,k\right).$
\end{defn}

Loosely speaking the definition of local ubiquity says that the set
$\Delta(\rho,t)$ locally approximates the underlying space
$\mb{I}^{mn}$ in terms of the Lebesgue measure. The function $\rho$,
will be referred to as the {\em ubiquity function}. The actual
values of the constants $\kappa$ and $k$  in the above definition
are irrelevant, it is their existence that is important.  In
practice the  local ubiquity of a system can be established using
standard arguments concerning the distribution of the resonant sets
in $\mb{I}^{mn}$, from which the function $\rho$ arises naturally.

Clearly if $ | \Delta\left( \rho ,t\right)|_{mn} \to 1 \text{
as }t\to{}\infty$
then $\left( \Re ,\beta \right) $ is
locally--ubiquitous. To see this let $B$ be any ball and assume
without loss of generality that \ $|B|_{mn} =\epsilon
>0.$ Then for t sufficiently large,
$$|\Delta\left( \rho ,t\right)|_{mn}
> 1 -\epsilon /2.$$ Hence $|B\cap
\Delta \left( \rho ,t\right)|_{mn}\geq \epsilon /2$ as required.

Given a positive real number $k>1$ a function $f$ will be said to be
\textsl{k-regular} if there exists a positive constant $\lambda <1$
such that for $t$ sufficiently large $$ f(k^{t+1}) \leq \lambda
f(k^t).
$$

Finally, we set $\gamma=\mathrm{dim}(R_{\mathbf{q}})$, the common (Euclidean) dimension of the resonant sets $R_{\mathbf{q}}$.

The following theorem is a simplified version of Theorem 1 from \cite{geo}.

\begin{theorem}[BV]
\label{BV}
Suppose that $(\Re ,\beta)$ is locally
ubiquitous relative to $(\rho, k)$ and $\psi$ is an
approximation function. Let $f$ be a dimension function such that
$r^{-\delta}f(r)$ is monotonic. Furthermore suppose that
$r^{-\gamma}f(r)$ is increasing and $\rho$ is $k$--regular. Then
\begin{equation}\label{sumcon}
\mcal{H}^f(W_0(m,n;\psi)) =  \mcal{H}^f(\mb{I}^{mn}) \quad
\textrm{if} \quad \sum_{n=1}^\infty
\frac{f(\Psi(k^t))\Psi(k^t)^{-\gamma}}{\rho(k^t)^{\delta-\gamma}}=\infty.
\end{equation}
\end{theorem}

\subsection{Slicing}\label{auxi}

We now state a result which is the crucial
key ingredient in the proof of Theorem \ref{Mumtaz2}.
The result was used in \cite{slicing} to
prove the Hausdorff measure version of the W. M. Schmidt's
inhomogeneous linear forms theorem in metric number theory. The authors refer to the technique as
``slicing''. We will merely state the result. For a more detailed discussion and proof see
\cite{slicing} or \cite{Mat}.
However, before we do state the theorem it is necessary to introduce a little notation.

Suppose that $V$ is a linear subspace of
$\mb{I}^k$, $V^{\perp}$ will be used to denote the linear subspace of
$\mb{I}^k$ orthogonal to $V$. Further $V+a:=\left\{v+a:v\in V\right\}$
for $a\in V^{\perp}$.

\begin{lemma}\label{slicing} Let $l, k \in \mb{N}$ be such that $l\leq k$ and let $f \;\textrm{and}\; g:r\to
r^{-l}f(r)$ be dimension functions. Let $B\subset \mb{I}^k$ be a
\emph{Borel} set and let $V$ be a $(k-l)$--dimensional linear
subspace of $\mb{I}^k$. If for a subset $S$ of $V^{\perp}$ of
positive $\mcal{H}^l$
measure$$\mcal{H}^{g}\left(B\cap(V+b)\right)=\infty \hspace{.5cm}
\forall \; b\in S,$$ \noindent then $\mcal{H}^{f}(B)=\infty.$
\end{lemma}

We are now in a position to begin the proofs of Theorems~\ref{thm1} and \ref{Mumtaz2}.

\section{The Proof of Theorem \ref{thm1}}\label{sthm1}

As stated above, the proof of Theorem~\ref{thm1} is split into two parts; the convergence case and the divergence case. We begin with the convergence case as this is more straightforward than the divergence case.

\subsection {The Convergence Case}\label{con}

Recall that in the statement of Theorem~\ref{thm1}  we
assumed that $m>n$ and we imposed some conditions on the dimension function $f$. As it turns out these conditions are not needed in the convergence case and we can state and prove a much cleaner result which has the added benefit of also implying the convergence case of Theorem \ref{Mumtaz2}.

\begin{theorem}\label{conv}
Let $\psi$ be an approximating function and let $f$ be a
dimension function. If
\[
\sum \limits_{r=1}^{\infty}f(\Psi (r))\Psi(r)^{-(m-1)n}r^{m-1}<\infty,
\]
then
\[
\mcal{H}^f\left(W_0\left( m,n;\psi
\right)\right)=0.
\]
\end{theorem}

Obviously Theorem~\ref{conv} implies the convergence cases of
Theorems~\ref{thm1} and \ref{Mumtaz2}.

\begin{proof}
To prove Theorem~\ref{conv} we make use of
the natural cover of $W_0\left( m,n;\psi \right)$ given by
Equations~\eqref{limsup1} and \eqref{limsup2}.
It follows almost immediately that for each $N\in \mb{N}$ the family%
\begin{equation*}
\left\{ \underset{R_{\mbf{q}}:\left\vert \mbf{q}\right\vert =r}{\bigcup }\Delta\left( R_{\mbf{q}},\Psi (\vert\mbf{q}%
\vert )\right) :r=N,N+1,\dots\right\}
\end{equation*}
is a cover for the set $W_0\left( m,n;\psi \right) $. That is
\[
W_0\left( m,n;\psi \right) \subset \underset{r>N}{{\bigcup
}}\underset{|\mbf{q}| =r}{{\bigcup }}\Delta( R_{\mbf{q}},\Psi(|\mbf{q}%
|))
\]
for any $N\in\mathbb{N}$.

Now, for each resonant set $R_{\mbf{q}}$ let $\Delta (q)$ be a collection of $mn$-%
dimensional closed hypercubes $C$ with disjoint interiors and side length
$\Psi (|\mbf{q}|)$ such that
\[
 C{}\bigcap 
\bigcup_{|\mathbf{q}|=r}
\Delta( R_{\mbf{q}},\Psi(|\mbf{q}|))
\neq \emptyset
\]
and
\[
\Delta(R_{\mbf{q}},\Psi (|
\mbf{q}|)) \subset \underset{C\in \Delta
(q)}{{\bigcup }}C.
\]
Then
\[
\# \Delta(q)\ll(\Psi(|\mbf{q}|))
^{-(m-1)n}.
\]
were $\#$ denotes cardinality.

Note that
\begin{equation*}
W_0\left( m,n;\psi \right) \subset \underset{r>N}{{\bigcup
}}\underset{\left\vert \mbf{q}\right\vert =r}{{\bigcup
}}\Delta\left( R_{\mbf{q}},\Psi (\left\vert
\mbf{q}\right\vert)\right) \subset \underset{r>N}{{\bigcup
}}\underset{\Delta (q):\left\vert \mbf{q}\right\vert =r}{{\bigcup
}}\underset{C\in \Delta (q)}{{\bigcup }}C.
\end{equation*}

It follows on setting $\rho(N)=\psi(N)$ that
\begin{eqnarray*}
\mcal{H}_{\rho}^f\left(W_0\left( m,n;\psi \right)\right)   &\leq &%
\sum\limits_{r>N} \sum\limits_{\Delta (q):\left\vert
\mbf{q}\right\vert =r} \sum \limits_{C\in \Delta (q)}f(\Psi
(\left\vert
\mbf{q}\right\vert )) \\
&\ll &\sum \limits_{r>N} r^{m-1}f\left(\Psi (r)\right)\Psi (r)
^{-(m-1)n}\to 0 \hspace{.5cm}as \hspace{.5cm}\rho\to 0,
\end{eqnarray*}
and thus from the definition of $\mcal{H}^{f}$--measure that
$\mcal{H}^f(W_0\left( m,n;\psi \right))=0$, as required.
\end{proof}

\subsection{The Divergence Case}\label{div}

When $m>n$, the divergence part of Theorem~\ref{thm1} relies on the notion of ubiquity and primarily Theorem~\ref{BV}.
To use ubiquity we must show that $\left( \Re ,\beta
\right) $ is locally--ubiquitous with respect to $\left( \rho
,k\right)$ for a suitable ubiquity function $\rho$. For the sake of simplicity we fix  
$k=2$.

To establish ubiquity we need two technical lemmas. The first of which is due to Dickinson~\cite{Dickinson} and is an analogue of
Dirichlet's theorem. The second is a slight modification again of a result of Dickinson from the same paper. The key difference being the introduction of a function $\omega$ instead of
$\log$. We prove only the second result here and the reader is referred to the previously mentioned paper for the proof of Lemma~\ref{dir}.
\begin{lemma}\label{dir}
  For each $X\in \mb{I}^{mn}$, there exists a non-zero integer vector
$\mbf{q}$ in $ \mb{Z}^{m}$ with $\left\vert \mbf{q}\right\vert \leq
2^{t}\left( t\in \mb{N} \right) $ such that
\begin{equation*}
\left\vert \mbf{q}X\right\vert <m\left( 2^{t}\right)
^{-\frac{m}{n}+1}.
\end{equation*}
\end{lemma}

\begin{lemma}
Let $\omega$ be a positive real increasing function such that $\frac{1}{\omega \left( t\right) }%
 \rightarrow 0$ as $t\rightarrow \infty $ and such that for any $C>1$
 and $t$ sufficiently large $\omega \left( 2t\right) <C\omega \left(
 t\right) .$ The the family $(\Re,\beta)$ is locally ubiquitous with respect to
the function $\rho : \mb{N} \rightarrow \mb{R} ^{+}$ where $\rho(t)=m(2^t)^{-\frac{m}{n}
}\omega(t)$.
\end{lemma}

\begin{proof}
Throughout this proof $\mbf{q}$ will refer to those integer vectors which
satisfy the conclusion of Lemma~\ref{dir}. Note that a simple calculation will establish the fact that $\rho$ is $2$-regular for $t$ sufficiently large. Define now the set $E(t)$ where
\[
E(t) =\{ X\in \mb{I}^{mn}:|\mbf{q}|<%
\frac{2^{t}}{\omega(t) }\}
\]
and
\[
\Delta(t) = \{ X\in \mb{I}^{mn}: |
X-\partial \mb{I}^{mn} | \geq 2^{-t}\} \setminus
E(t),
\]
$\partial \mb{I}^{mn}$ denotes the boundary of the set $\mathbb{I}^{mn}$.

\noindent Then%
\begin{equation*}
E(t) \subseteq \underset{1\leq r\leq \frac{2^{t}}{\omega
\left(
t\right) }}{{\bigcup }}\ \underset{\left\vert \mbf{q}\right\vert =r}{%
{\bigcup }}\left\{ X\in \mb{I}^{mn}:\left\vert \mbf{q}X\right\vert
<m\left( 2^{t}\right) ^{\frac{-m}{n}+1}\right\}.
\end{equation*}

\noindent Therefore%
\begin{eqnarray*}
\left\vert E\left( t\right) \right\vert_{mn} &\leq&\sum \limits_{1\leq r\leq \frac{%
2^{t}}{\omega \left( t\right) }} \sum \limits_{\left\vert \mbf{q}%
\right\vert =r}\frac{m^n\left( 2^{t}\right) ^{-m+n}}{\left\vert
\mbf{q}\right\vert _{2}^{n}} \\
&\ll &\left( 2^{t}\right) ^{-m+n}\sum \limits_{1\leq r\leq
\frac{2^{t}}{\omega
\left( t\right) }} r^{m-n-1} \\
&\ll &\left( 2^{t}\right) ^{-m+n}\frac{2^{t}}{\omega \left( t\right)
}\left(
\frac{2^{t}}{\omega \left( t\right) }\right) ^{m-n-1} \\
&=&\left( \omega \left( t\right) \right) ^{-m+n}.
\end{eqnarray*}

\noindent Therefore, since $m>n, \lim\limits_{t\rightarrow \infty
}\left\vert E\left(
t\right) \right\vert_{mn} \rightarrow 0$ and $\underset{t\rightarrow \infty }{\lim}%
\left\vert \mb{I}^{mn}\backslash \Delta \left( t\right)
\right\vert_{mn} \rightarrow 0. $ Now to show that $\left\vert
\left( \Delta\left( \rho ,t\right) \right) \right\vert_{mn}
\rightarrow 1$ as $t\rightarrow \infty ,$ it would be enough to show
that $\Delta \left( t\right) \subseteq \Delta \left( \rho ,t\right)
.$ For this let $X\in \Delta \left( t\right) \Rightarrow X\notin
E\left( t\right) $ and let $\overset{\sim }{\mbf{q}}$ be from lemma
\ref{dir},
\begin{equation*}
\frac{2^{t}}{\omega \left( t\right) }\leq \left\vert \overset{\sim }{\mbf{%
q}}\right\vert \leq 2^{t}.
\end{equation*}

\noindent By definition $\left\vert \overset{\sim
}{\mbf{q}}\right\vert =\left\vert \overset{\sim }{q}_{i}\right\vert
$ for some $1\leq i\leq m.$ Let $\delta _{j}=\frac{-\overset{\sim
}{\mbf{q}}\cdot \mbf{x}^{\left( j\right) }}{\left\vert \overset{\sim
}{q}_{i}\right\vert },j=1,2\dots,n$ so that $\overset{\sim
}{\mbf{q}}\cdot \left( \mbf{x}^{\left(
j\right) }+\delta _{j}\mbf{e}^{\left( i\right) }\right) =0,$ where $%
\mbf{e}^{\left( i\right) }$ denotes the i'th basis vector. Also%
\begin{equation*}
\left\vert \delta _{j}\right\vert =\left\vert \frac{-\overset{\sim }{\mbf{%
q}}\cdot \mbf{x}^{\left( j\right) }}{\left\vert \overset{\sim }{q}%
_{i}\right\vert }\right\vert \leq m\left( 2^{t}\right)
^{\frac{-m}{n}}\omega \left( t\right).
\end{equation*}%
\noindent Therefore $U=\left( \mbf{x}^{\left( j\right) }+\delta _{j}\mbf{e}%
^{\left( i\right) }\right) =\left( \mbf{x}^{\left( 1\right) }+\delta _{1}%
\mbf{e}^{\left( i\right) },\mbf{x}^{\left( 2\right) }+\delta _{2}%
\mbf{e}^{\left( i\right) }\dots,\mbf{x}^{\left( n\right) }+\delta _{n}%
\mbf{e}^{\left( i\right) }\right) $ is a point in $R_{\overset{\sim }{%
\mbf{q}}}$ and $\left\vert X-U\right\vert \leq m\left( 2^{t}\right) ^{%
\frac{-m}{n}}\omega \left( t\right) =\rho \left( t\right) .$ Hence
$X\in
 \Delta \left( \rho ,t\right)
 =\underset{%
2^{t-1}<\left\vert \mbf{q}\right\vert \leq 2^{t}}{{\bigcup
}}\Delta\left( R_{\mbf{q}},\rho \left( t\right) \right) $ so that
\begin{equation*}
\left\vert \left( \Delta\left( \rho ,t\right) \right)
\right\vert_{mn} \rightarrow 1\text{ \ as \ }t\rightarrow \infty.
\end{equation*}

\end{proof}

We are now almost in position to apply Theorem~\ref{BV}. To this end consider the sum
\[
\label{sumcomp}
\sum \limits_{t=1}^{\infty}f\left(\Psi \left( k^{t}\right) \right)
\left(\frac{\Psi(k^t)}{\rho(k^t)}\right)^{\delta-\gamma}\Psi(k^t)^{-\delta},
\]
which is comparable to
\[
\label{sumcomp2}
\sum_{t=1}^{\infty}f(\Psi ( 2^{t}) )
\Psi(2^t)^{-(m-1)n}(2^t)^{m}\omega(t)^{-n}.
\]

Assuming that $\psi:\mb{R}^{+}\rightarrow\mb{R}^{+}$ is a monotonic function, $\alpha,
\beta \in\mb{R}$ and $k>1.$ Let $f$ be a dimension function. It is straightforward to show that the
convergence or divergence of the sums
\begin{equation*}
\sum\limits_{t=1}^{\infty} k^{t\alpha }f(\psi \left(
k^{t}\right))\psi \left( k^{t}\right)^{\beta} \text{ \ \ \ and \ \ \
}\sum \limits_{r=1}^{\infty}r^{\alpha -1}f(\psi \left( r\right))\psi
\left( r\right)^{\beta}
\end{equation*}
coincide. By virtue of this fact, the sum in Equation~\eqref{sumcomp} is the same as
\begin{equation}\label{sumequ}
\sum \limits_{r=1}^{\infty}f\left( \Psi \left( r\right) \right)
\Psi\left( r\right) ^{-(m-1)n}r^{m-1}\omega \left( r\right) ^{-n}.
\end{equation}

To obtain the precise statement of the Theorem
\ref{thm1} we need to remove the $\omega $ factor from the above. To do this we
choose $\omega $ in such a way that the sum
\[
\sum \limits_{r=1}^{\infty}f\left( \Psi \left(
r\right) \right) \Psi\left( r\right) ^{-(m-1)n}r^{m-1}\omega \left(
r\right) ^{-n}
\]
will converge (\textit{respec.} diverge) if and only if the sum
\begin{equation}\label{eq:sum2}
\sum \limits_{r=1}^{\infty}f\left( \Psi \left( r\right) \right)
\Psi\left( r\right) ^{-(m-1)n}r^{m-1}
\end{equation}
converges (\textit{respec.} diverges).
This is always possible. Firstly, note that if the sum in Equation~\eqref{sumequ} diverges then so does the sum in Equation~\eqref{eq:sum2}.
On the other hand and if the sum in Equation~\eqref{eq:sum2} diverges, then we
can find a strictly increasing sequence of positive integers
$\{r_i\} _{i\in \mb{N}}$ such that
\[
\sum_{r_{i-1}\leq r\leq r_{i}}^{\infty}f( \Psi(r) )\Psi(r)^{-(m-1)n}r^{m-1}>1
\]
and $r_{i}>2r_{i-1}$. Now simply define $\omega$ be the step function $%
\omega(r) =i^{\frac{1}{n}}$ for $r_{i-1}\leq r\leq
r_{i}$ and $\omega$ satisfies the required properties. 

This completes the proof of Theorem~\ref{thm1}.

\section{Proof of Theorem \ref{Mumtaz2}}\label{proofThm2}

In view of Theorem~\ref{conv} we need only prove the
divergence part of Theorem~\ref{Mumtaz2}. The proof will be split into two
sub-cases. The first, which we refer to as the ``infinite measure'' case, is for dimension functions $f$ such that $r^{-(m-1)(n+1)}f(r)\to\infty$. 
The second case corresponds to $f$ which satisfy $r^{-(m-1)(n+1)}f(r)\to{}C$ for some constant $C>0$, in which case the measure is comparable to $(m-1)(n+1)$-Lebesgue measure 
and we call this case the ``finite measure'' case. 

We begin the proof of Theorem~\ref{Mumtaz2} with the
key observation that if $m\leq n$,
$W_0(m,n;\psi)$ lies in a manifold of dimension at most $(m-1)(n+1)$. 

Consider first the case when $m=n$.
Take any $X\in{}W_0(m,m;\psi)$, then
the column vectors of $X$ 
are linearly dependent. To prove this, assume to the contrary, that the column vectors 
are linearly independent.
Since $X$ is a member of $W_0(m, n; \psi)$ there exists
infinitely many $\mbf{q}$ such that
\[
|\mbf{q}X|<\psi({\mbf{|q|}}).
\]
Setting $\mbf{q}X=\theta$ where $
|\theta|<\psi({\mbf{|q|}})$, as all column vectors are linearly
independent 
$X$ is invertible. Thus
\[
|\mbf{q}|=|\theta X^{-1}|
\]
and it follows that
\[
1\leq |\mbf{q}|=|\mbf{q} X X^{-1}|\leq C_2(X)\psi(|\mbf{q}|)\to 0
\text{\quad{}as\quad} |\mbf{q}|\to\infty.
\]
Which is clearly impossible. Therefore the column vectors of $X$ must be
linearly dependent and so $\det{}X=0$. This in turn implies that $X$ lies on some surface defined by the multinomial equation $\det{}Y=0$ where $Y\in\mathbb{I}^{m^2}$. As this equation defines a co-dimension $1$ manifold in $\mathbb{I}^{m^2}$, at most
$m^{2}-1$ independent variables are needed to fully specify $X$.

This above argument is essentially that needed to prove the more general case when $m\leq{}n$. We prove the result only for the case when $n=m+1$ as the general case follows with a straightforward modification of the argument given.  Given any $X\in{}W(m,m+1;\psi)$. Thinking of $X$ as an $m$ by $m+1$ matrix as above. Let $x$ be the first column vector, $x_2$ the $(m+1)$--th column vector and $X^\prime$ the $m$ by $m-1$ matrix formed by taking the remaining $m-1$ columns of $X$. Further let $X_1$ be the $m\times{}m$ matrix with first column $x$ and remaining columns made up of $X^\prime$. Similarly let $X_2$ be the $m\times{}m$ matrix with first $m-1$ columns the same as $X^\prime$ and final column $x_2$. Using the same argument as above we  claim that both these sub-matrices of $X$ are in fact non-invertible and so each sub-matrix lies on a co-dimension $1$ manifold $\Gamma_i$ determined by the equation $\det{}Y_i=0$ with $i=1,2$. Here $Y_1$ is an $m\times{}m$ matrix consisting of all but the final $m$ variables of an arbitrary element $Y\in\mathbb{I}^{m(m+1)}$ and $Y_2$ is similarly defined but the first $m$ variables are now removed. Now $X$ must lie in the intersection of the two manifolds $\Gamma_1$ and $\Gamma_2$, say $\Gamma$. This is a co-dimension $2$ manifold and the result is proved. 

With the above observation in mind we begin the proof of Theorem~\ref{Mumtaz2} in earnest by defining the set:
\begin{equation}
W_0\left( m,n;c\psi \right) :=\left\{ X\in \mb{I}
^{mn}:\left\vert \mbf{q}X\right\vert <c\psi (\left\vert \mbf{q}%
\right\vert )\text{ for i.m.}\text{ \ }\mbf{q}\in
\mb{Z}^{m}\setminus\{\mbf{0}\}\right\},
\end{equation}
\noindent where $c=\max(\frac{m-1}{2},1).$  It is clear then that 
$W_0\left(m,n;\psi \right)\subseteq W_0\left( m,n;c\psi \right).$ 

Let $A$ be
the set of points of the form
\[
\left(X^{(1)},X^{(2)},\dots,X^{(m-1)},\sum\limits_{j=1}^{m-1} a_{j}^{(1)}{X^{(j)}}
,\dots, \sum\limits_{j=1}^{m-1} a_{j}^{(n-m+1)}X^{(j)}\right),
\]
where 
\[
\left(X^{(1)},X^{(2)},\dots,X^{(m-1)}\right)\in
W_0(m,m-1;\psi)
\]
and $a^{(i)}_{j}\in\left(\frac{-1}{2}, \frac{1}{2}\right)$ for $1\leq{}i\leq(n-m+1)$.
Note that
\begin{align*}
\left\vert\mbf{q}\cdot\sum\limits_{j=1}^{m-1}a_{j}^{(i)}X^{(j)}\right\vert
&=\left\vert\sum\limits_{j=1}^{m-1}a_{j}^{(i)}\mbf{q}\cdot
X^{(j)}\right\vert\\
 &\leq \sum\limits_{j=1}^{m-1}\vert a_{j}^{(i)} \vert
\vert \mbf{q} \cdot X^{(j)}\vert \\
&\leq \left(\sum\limits_{j=1}^{m-1}\vert
a_{j}^{(i)}\vert\right)\psi(|\mbf{ q}|)\\ &\leq
c\psi(|\mbf{q}|)\hspace{1cm} 1\leq i\leq(n-(m-1)),
\end{align*}
and it follows that $A\subseteq{}W(m,n; c\psi)$.

Now define the function
\[
\eta:W_0(m,m-1,\psi)\times
\left(\frac{-1}{2}, \frac{1}{2}\right)^{(n-(m-1))(m-1)}\to A
\]
by
\begin{eqnarray*}
\eta\left(X^{(1)},X^{(2)},\dots,X^{(m-1)},
a_{1}^{1},\dots,a_{m-1}^{1},\dots,a_{1}^{(n-(m-1))},\dots,a_{m-1}^{(n-(m-1))}\right)&=&\\
\left(X^{(1)},X^{(2)},\dots,X^{(m-1)},\sum\limits_{j=1}^{m-1}
a_{j}^{(1)}X^{(j)},\dots,\sum\limits_{j=1}^{m-1}
a_{j}^{(n-m+1)}X^{(j)}\right).
\end{eqnarray*}
Note that $\eta$ is surjective and that the vectors  $X^{j}$, for
$j=1,\dots,m-1$, are linearly independent. This ensures that $\eta$ is well defined,
one-to-one and the Jacobian, $J(\eta)$, of $\eta$ is of maximal rank. The function
$\eta$ is therefore an embedding and its range is
diffeomorphic to $A$. This in turn implies that $\eta$ is
(locally) bi-Lipschitz.

\subsection{The Infinite Measure Case}

As mentioned above the proof of Theorem~\ref{Mumtaz2} is split into two parts. In this section we concentrate on the infinite measure case which can be deduced from the following lemma.
\begin{lemma}
    \label{lem:infmeslem}
 Let $\psi$ be an approximating function and let $f$ and $g:r\to r^{-(n-(m-1))(m-1)}f(r)$ be dimension
functions with $r^{-(m-1)(n+1)}f(r)\to\infty$ as $r\to{}0$. 
Further, let  $r^{-m(m-1)}g(r)$ be monotonic and $r^{-(m-1)^2}g(r)$ be increasing. If
\[
  \sum_{r=1}^{\infty}f(\Psi
    (r))\Psi(r)^{-(m-1)n}r^{m-1}=\infty,
\]
then
\[
\mcal{H}^{f}(A)=\infty.
\]
\end{lemma}
\begin{proof}
As $\eta$ is bi-Lipschitz, we have that 
\begin{eqnarray*} \mcal{H}^{f}(A)
&=&\mcal{H}^{f}\left(\eta\left(W_0(m,m-1,\psi)\times
\mb{I}^{(n-(m-1))(m-1)}\right)\right)\notag\\&\asymp&\mcal{H}^{f}\left(W_0(m,m-1,\psi)\times
\mb{I}^{(n-(m-1))(m-1)}\right).
\end{eqnarray*}

The proof relies on the slicing technique of Lemma~\ref{slicing}.
Let $B:=W_0(m,m-1;\psi)\times
\mb{I}^{(n-(m-1))(m-1)}\subseteq \mb{I}^{(m-1)(n+1)}$
and $V$ be the space $\mathbb{I}^{m(m-1)}\times{}\{0\}^{(m-1)(n+1-m)}$.
As $W_0(m,m-1;\psi)$ is a $\limsup$ set, $B$ is a
Borel set. 
We know that
$\dim W_0(m,m-1; \psi) = m(m-1)$ by Theorem~\ref{thm1} and this means that $W_0(m,m-1;\psi)$ is dense in $\mathbb{I}^{m(m-1)}$. 
Let $S:=\{0\}^{m(m-1)}\times\mb{I}^{(n+1-m)(m-1)}$. Clearly $S$ is a 
subset of $V^\perp$,
and further it has positive $\mcal{H}^{(n-(m-1))(m-1)}$-measure. Now for each 
$b\in{}S$
\begin{eqnarray*}
\mcal{H}^{g}\left(B\cap(V+b)\right)&=&\mcal{H}^{g}\left((W_0(m,m-1;\psi)\times
\mb{I}^{(n-(m-1))(m-1)})\cap(V+b)\right)\notag\\&=&\mcal{H}^{g}(\widetilde{W}_0(m,m-1,\psi)+b)\notag\\&=&\mcal{H}^{g}\left(\widetilde{W}_0(m,m-1;\psi)\right)
\label{eqn:gmesB},
\end{eqnarray*}
where $\widetilde{W}_0(m,m-1;\psi)=W(m,m-1;\psi)\times\{0\}^{(n+1-m)(m-1)}$. Thus the $g$-measure of 
$\widetilde{W}_0(m,m-1;\psi)$ conicides with the $g$-measure of $W(m,m-1;\psi)$ and 
Now applying Theorem~\ref{thm1} 
with $n=m-1$ implies that 
$\mcal{H}^{g}\left(B\cap(V+b)\right)=\infty$
if
\[
\sum_{r=1}^\infty
r^{m-1}g(\Psi(r))\Psi(r)^{-(m-1)^2}=\infty.
\]
Applying Lemma~\ref{slicing}, we have 
$\mcal{H}^{f}(A)=\infty$
if
\[
\sum_{r=1}^\infty
r^{m-1}g(\Psi(r))\Psi(r)^{-(m-1)^2}=\infty
\]
and we conclude that $\mcal{H}^{f}(A)=\infty$ 
if
\[
 \sum_{r=1}^\infty
f(\Psi(r))\Psi(r)^{-(m-1)n}r^{m-1}=\infty,
\]
as required.
\end{proof}

We can now compete the proof of Theorem~\ref{Mumtaz2}. 
As
$A\subseteq W_0(m,n;c\psi)$,
$\mcal{H}^f(A)=\infty$ implies that $\mcal{H}^f(W_0(m,n;c\psi))=\infty$ and we need only show 
that the value of the constant $c$ is
irrelevant. Recall that $c\geq{}1.$ 
For convenience let $\psi_c(r):=\frac{\psi(r)}{c}$,
$\Psi_c(r):=\frac{\Psi(r)}{c}$, $\sum:=\sum_{r=1}^{\infty}
f(\Psi(r))\Psi(r)^{-(m-1)n}r^{m-1}$ and
$\sum_c:=\sum_{r=1}^{\infty}f(\Psi_c(r))\Psi_c(r)^{-(m-1)n}r^{m-1}$
Since $r^{-(m-1)(n+1)}f(r)$ is decreasing it follows that
\[
\infty=\sum\leq c_1\sum{}_c
\]
where $c_1=c^{-(m-1)(n+1)}$. Therefore
$\sum_c=\infty$ if
$\sum=\infty$ and we have $\mcal{H}^{f}(W_0( m,n;c\psi_c)=\infty$ if 
$
\sum_{r=1}^\infty
f(\Psi(r))\Psi(r)^{-(m-1)n}r^{m-1}=\infty.
$
Finally, it follows that $\mcal{H}^{f}(W_0( m,n;\psi)=\infty$ 
if
$
\sum_{r=1}^\infty
f(\Psi(r))\Psi(r)^{-(m-1)n}r^{m-1}=\infty,
$
as required.

\subsection{Finite measure case} 

We now come onto the case where $r^{-(m-1)(n+1)}f(r)\to C$ as $r\to 0$ and $C>0$ is finite.  
In this case 
$\mcal{H}^f$ is comparable to 
$(m-1)(n+1)$--dimensional
Lebesgue measure. 
Note that in this case the dviergence of the sum
\[
 \sum_{r=1}^\infty{}f(\Psi(r))\Psi(r)^{-(m-1)n}r^{m-1}
\]
is in direct correspondence with that of the sum
\[
  \sum_{r=1}^\infty{}\psi^{m-1}(r).
\]

We begin with the following general lemma, the proof of which we leave to the reader. 
\begin{lemma}\label{fulllemma}
 Suppose that $L\subset\mb{R}^l$, $M\subset\mb{R}^k$ and $\eta:L\to M$ is an onto bi--Lipschitz transformation. That is there exists constants $c_1$ and $c_2$ with
 $0<c_1\leq c_2<\infty$, such that
 \[
    c_1 d_L(x,y) \leq d_M(\eta(x),\eta(y)) \leq c_2{}d_L(x,y)
  \]
  for any $x,y\in{}L$ where $d_L$ and $d_M$ are the respective metrics on $L$ and $M$.
  Then for any $C\subseteq L$,  with $|C|_L = 0$, we have $|\eta(C)|_M =0$ and for any $C^\prime\subseteq L$ with $|L\setminus C^\prime|_L=0$, 
  $|\eta(L\setminus C^\prime)|_M \ = \ 0$ where $|\cdot|_L$ (\textit{respec} $|\cdot|_M$) denotes induced 
  measure on $L$ (\textit{respec} $M$). 
  
  That is $\eta$ preserves a $zero-full$ law.
\end{lemma}

In applying Lemma~\ref{fulllemma}, we first need to show that 
$W_0(m,m-1;\psi)\times \mb{I}^{(n-m+1)(m-1)}$ has $full$ Lebesgue
measure in $\mb{I}^{(m-1)(n+1)}$. Theorem~\ref{thm1} implies that 
$|W_0(m,m-1;\psi)|_{m(m-1)} = 1$ if
$\sum_{r=1}^{\infty}\psi(r)^{m-1}=\infty$ and a straightforward application of  
Fubini's Theorem gives
\[
|W_0(m,m-1;\psi)\times
\mb{I}^{(n-(m-1))(m-1)}|_{(m-1)(n+1)}= 1
\]
as the Lebesgue measure of a product of two sets is simply the product of the measures of the two sets. 
It follows then that
$W_0(m,m-1;\psi)\times \mb{I}^{(n-m+1)(m-1)}$ is full in $\mathbb{I}^{(m-1)(n+1)}$,
as required.  

It remains to prove that $A$, the image of $W_0(m,m-1;\psi)\times \mb{I}^{(n-m+1)(m-1)}$ under $\eta$ is full in $\Gamma$. To do this we use local charts on $\Gamma$. As $\Gamma$ is an $(m-1)(n+1)$-dimension smooth manifold, we know that there is a countable atlas for $\Gamma$. Take any chart in the atlas, say $(O,\nu)$ where $O$ is an open set in $\mathbb{R}^{(m-1)(n+1)}$ and $\nu$ is the (local) diffeomorphism from $O$ to $\Gamma$. Now, $\eta$ is invertiable and $\eta^{-1}(\nu(O))$ is in $\mathbb{I}^{(m-1)(n+1)}$. We have just shown that $W_0(m,m-1;\psi)\times \mb{I}^{(n-m+1)(m-1)}$ has full measure and so therefore must its intersection with $\eta^{-1}(\nu(O))$. It follows then that $\eta$ of this intersection must have the same induced measure on $\Gamma$ as $\nu(O)$ does. We can repeat this argument for each element of the atlas of $\Gamma$ and it follows that $\eta(A)$ must be full in $\Gamma$ as required. 

This completes the proof of Theorem~\ref{Mumtaz2}.

\end{document}